\documentclass[reqno,12pt]{amsart}

\usepackage{amscd,amssymb,comment,epic,eepic,euscript}
\usepackage{graphicx}
\usepackage{enumitem}
\usepackage[initials,alphabetic]{amsrefs}
\usepackage{array}
\usepackage{longtable}

\usepackage[usenames,dvipsnames]{xcolor}

\usepackage{color,titletoc}

\renewcommand{\emptyset}{\varnothing}

\def\bes{\begin{equation*} }
\def\ees{\end{equation*} }

\def\gS{\mathfrak S}

\def\al{\alpha}
\def\be{\beta}
\def\ga{\gamma}

\renewcommand{\thetable}{\arabic{table}}

\theoremstyle{plain}
\newtheorem{thm}{Theorem}[section]
\newtheorem{cor}[thm]{Corollary} 
\newtheorem{lemma}[thm]{Lemma} 
\newtheorem{lem}[thm]{Lemma} 
\newtheorem{prop}[thm]{Proposition}
\newtheorem{conj}[thm]{Conjecture}

\theoremstyle{remark}

\newtheorem{remark}[thm]{Remark}

\theoremstyle{definition}

\DeclareMathOperator{\GL}{GL}
\DeclareMathOperator{\Tr}{Tr}
\DeclareMathOperator{\sign}{sign}

\newcommand{\C}{{\mathbb C}}
\newcommand{\R}{{\mathbb R}}
\newcommand\diag{\mathrm{diag}}

\newcommand\Image{\mathrm{Im}}
\newcommand\lspan{\mathrm{span}\,}
\newcommand{\N}{{\mathbb N}}

\renewcommand{\emptyset}{\varnothing}

\def\bes{\begin{equation*} }
\def\ees{\end{equation*} }

\def\beq{ \begin{equation} }
\def\eeq{ \end{equation} }

\def\bep{\begin{proof}}
\def\eep{\end{proof}}

\def\ben{\begin{enumerate}}
\def\een{\end{enumerate}}

\def\bet{\begin{theorem}}
\def\eet{\end{theorem}}

\def\bel{\begin{lemma}}
\def\eel{\end{lemma}}

\newcommand{\ax}{\langle x\rangle}
\newcommand{\axn}{\langle x_1,\ldots,x_n\rangle}
\newcommand{\csim}{\stackrel{\mathrm{cyc}}{\thicksim}}

\def\al{\alpha}
\def\be{\beta}
\def\ga{\gamma}

\def\de{\delta}

\begin{document}

\title[The Kaplansky-Lvov multilinear conjecture]{Instances of the  Kaplansky-Lvov  multilinear conjecture for polynomials of degree three}

\author[K. Dykema]{Kenneth J. Dykema$^{*}$}
\address{K.D., Department of Mathematics, Texas A\&M University,
College Station, TX 77843-3368, USA}
\email{kdykema@math.tamu.edu}
\thanks{\footnotesize ${}^{*}$Research supported in part by NSF grant DMS--1202660.}
\urladdr{https://www.math.tamu.edu/~ken.dykema/}

\author[I. Klep]{Igor Klep${}^{\dag}$}
\address{I.K., Department of Mathematics, 
The University of Auckland, Private Bag 92019, Auckland 1142, New Zealand}
\email{igor.klep@auckland.ac.nz}
\thanks{${}^\dag$Supported by the Marsden Fund Council of the Royal Society of New Zealand. Partially supported by the Slovenian Research Agency grants P1-0222, L1-4292 and L1-6722. Part of this research was done while the author was on leave from the University of Maribor.}
\urladdr{https://www.math.auckland.ac.nz/~igorklep/}

\date{\today}

\setcounter{tocdepth}{2}
\contentsmargin{2.55em} 
\dottedcontents{section}[3.8em]{}{2.3em}{.4pc} 
\dottedcontents{subsection}[6.1em]{}{3.2em}{.4pc}

\makeatletter

\begin{abstract}
Given a positive integer $d$, the Kaplansky-Lvov conjecture states that the set of values of a multilinear noncommutative polynomial $f\in \C\axn$ on the matrix algebra $M_d(\C)$ is a vector subspace. In this article
the technique of using one-wiggle families of Sylvester's clock-and-shift matrices is championed 
to  establish the conjecture for polynomials $f$ of degree three 
when $d$ is even or $d<17$.
\end{abstract}

\maketitle

\section{Introduction and the Statement of the Main Result}

Images of noncommutative polynomials play a fundamental role in noncommutative algebra and are a central topic of the theory of polynomial identities \cites{Pro73,Row80}. Another area where these objects play a prominent role is free analysis \cites{Voi10,KVV14},  especially its free real algebraic geometry branch \cite{Hel02}. This recent progress has led to a  surge of interest in images of noncommutative polynomials in matrix rings. A fundamental open problem in this regard is (cf.~\cite{Dn93}):

\begin{conj}[The Kaplansky-Lvov multilinear conjecture]
\label{conj:kap}
Let $f$ be a multilinear polynomial with complex coefficients, and let $d\in\N$. Then
the set of values of $f$ in $M_d(\C)$ is a vector space.
\end{conj}

Conjecture \ref{conj:kap} is stated in \cite{Dn93} for all fields, not just $\C$. We have, however, chosen to present the conjecture only over $\C$ as this is
where
our main interest lies.
Likewise many of the 
results from the literature cited below were proved for large classes of fields, but we shall only state their restrictions to $\C$. Incidentally, by general model theory, all our results presented over $\C$ are valid over
arbitrary algebraically closed fields of characteristic $0$.

If the set of values of a noncommutative polynomial $f$ is a vector subspace of $M_d(\C)$, then 
it is a Lie ideal (see e.g.~\cite{BK09}), and hence either $\{0\}$, $\C\cdot I_d$, $M_d(\C)\cap\ker\Tr=[M_d(\C),M_d(\C)]$ or $M_d(\C)$ by an old result of Herstein. 
Thus noncommutative polynomials can be classified based
on their (span of) values in $M_d(\C)$.
A very special instance of Conjecture \ref{conj:kap} is the case $f=[x,y]=xy-yx$, which is a classical result in matrix theory due to  Shoda \cite{Sho37} (see also Albert-Muckenhoupt \cite{AM57}): every traceless matrix is a commutator.
In \cite{KMR12}  Kanel-Belov, Malev and Rowen established Conjecture \ref{conj:kap} for $d=2$, i.e., for values in $2\times2$ matrices. 
In \cite{Spe13} \v Spenko proves Conjecture \ref{conj:kap} for \emph{Lie} polynomials (i.e., elements of a free Lie algebra) of degree $\leq4$. 
Mesyan \cite{Mes+} extends this to polynomials of degree 3 which are sums of commutators, while Buzinski and Winstanley \cite{BW13}
present an extension to multilinear sums of commutators of degree $4$. Further recent progress on images 
of multilinear polynomials is given in
 \cite{AEV15,CW16,LT16,MO16,KMR+,KMR-}.

Our main result establishes Conjecture \ref{conj:kap} for
polynomials $f$ of degree three when 
$d$ is even or $d<17$ is odd:

\begin{thm}\label{thm:main}
Let $f$ be a complex multilinear polynomial of degree three. 
\ben[label={\rm(\arabic*)}]
\item\label{it:thm1} If $d\in\N$ is even
then the image of $f$ in $M_d(\C)$ is a vector space.
 \item\label{it:thm2} If $d\in\N$  is odd and $d<17$,
then the image of $f$ in  $M_d(\C)$ is a vector space.
\een
\end{thm}

The main novelty in our approach is the use of 
the clock-and-shift matrices
first utilized by Sylvester \cite{Syl82}.
These matrices are
ubiquitous in mathematical physics (cf.~\cite{Wey50}*{Chapter IV, \S15} or \cites{BFSS97,CDS98})
and endow
$M_d(\C)$ with a group--with--cocycle structure (see 
Subsection \ref{subsec:sylvester} below for details).
 Our proofs are elementary though supported by computer calculations,
 and the presentation is essentially self-contained.

\subsection{Reader's guide}
The paper is organized as follows. We collect preliminaries and introduce notation in the next section. Then in Section \ref{sec:deg3all} we prove
Theorem \ref{thm:main}.
At the onset of the section (Subsection \ref{subsec:wiggle}) we present the main idea and the general strategy for the proof.
Subsection \ref{subsec:3.3} proves Theorem \ref{thm:main} for $d=3$,
Subsection \ref{subsec:17} completes the proof of Theorem \ref{thm:main}\ref{it:thm2}. In Subsection \ref{subsec:even} the theorem is
established for $d\equiv 2 \pmod 4$, and the proof finally concludes
in Subsection \ref{subsec:4}.

\subsection*{Acknowledgments}
The authors thank 
\v Spela \v Spenko
for discussions, and 
Jurij Vol\v ci\v c for reading
a preliminary version of this paper.

\section{Preliminaries}\label{sec:prelim}

In this section we fix notation and terminology, and gather a few basic results needed later in the paper.

\subsection{Words and polynomials}
Fix $n\in\N$ and let $x=(x_1,\ldots,x_n)$ denote an $n$-tuple of freely noncommuting variables. 
We write $\ax$ for the monoid freely
generated by $x$, i.e., $\ax$ consists of \textbf{words} in the $n$
letters $x_{1},\ldots,x_{n}$
(including the \emph{empty word $\emptyset$} which plays the role of the identity).
Let $\C\ax$ denote the  corresponding free algebra; its elements
are called
noncommutative (nc) polynomials.
An element of the form $aw$ where $0\neq a\in \C$ and
$w\in\ax$ is called a \textbf{monomial} and $a$ its
\textbf{coefficient}. Hence words are monomials whose coefficient is
$1$.
The
length of the longest word in a  polynomial $f\in\C\ax$ is the
\textbf{degree} of $f$ and is denoted by $\deg( f)$. The set
of all words of degree at most $k$ is $\ax_k$, and $\C\ax_k$ is the vector
space of all  polynomials of degree at most $k$.
A polynomial $f\in\C\ax$ is called \textbf{multilinear}  if it is linear in each of the variables 
$x_j$. If $f=pq-qp$ for some $p,q\in\C\ax$, then $f$ is a {\bf commutator}.

\begin{remark}\label{rem:cyceq}
It is straightforward to see that a polynomial $f=\sum_{\de\in\ax}a_\de\de$ is a sum of commutators if and only if for each 
$\tau\in\ax$,
$$\sum_{\de\csim\tau}a_\de=0.$$
Here  words $\de,\tau\in\ax$ are called
{\bf cyclically equivalent}, $\de\csim\tau$, if one is a cyclic
permutation of the other. Equivalently, $\de-\tau$ is a commutator in $\C\ax$.
\end{remark}

\subsection{Group--with--cocycle structure in $M_d(\C)$}\label{subsec:sylvester}

We will use extensively the standard group--with--cocycle structure in $M_d(\C)$.
Let $\omega=e^{2\pi i/d}$ and for integers $p$ and $q$, let
\begin{equation}\label{eq:upq}
u^p_q=y^qv^p,
\end{equation}
where
\begin{align*}
y&=\diag(1,\omega,\omega^2,\ldots,\omega^{d-1}) \\
v&=\left(\begin{matrix}0&1 \\ &\ddots&\ddots \\ &&\ddots&1 \\ 1&&&0\end{matrix}\right)
\end{align*}
are {\bf Sylvester's clock-and-shift permutation matrices}, respectively \cite{Syl82}.
Note that the superscript on $u$ in~\eqref{eq:upq} is simply a label, not an exponent,
while the superscripts on $v$ and $y$ are exponents.
Of course $u^{r}_{s}=u^p_q$ if both $r-p$ and $s-q$ are divisible by $d$ and we have
\beq\label{eq:keyRel}
u^{p_1}_{q_1}u^{p_2}_{q_2}=\omega^{p_1q_2}u^{p_1+p_2}_{q_1+q_2}\,.
\eeq
Furthermore,
$\{u^p_q\mid p,q\in\{0,\ldots,d-1\}\}$ is a linear basis for $M_d(\C)$.

\section{Proof of Theorem \ref{thm:main}}\label{sec:deg3all}

Let
\begin{equation}\label{eq:f}
f(x_1,x_2,x_3)=\sum_{\sigma\in \gS_3}a_\sigma x_{\sigma(1)}x_{\sigma(2)}x_{\sigma(3)}
\end{equation}
for $a_\sigma\in\C$
and let $d\in\N$, $d\ge2$.
(Here $\gS_n$ is used to denote the symmetric group of permutations on
$\{1,\ldots,n\}$.)
We are interested in the image of $M_d(\C)^3$ under $f$, which we will denote $\Image f$.
Clearly, if $\sum_{\sigma\in \gS_3}a_\sigma\ne0$, then $\Image f=M_d(\C)$.
Henceforth, assume
\begin{equation}\label{eq:sum0}
\sum_{\sigma\in \gS_3}a_\sigma=0
\end{equation}
but that $a_\sigma\ne0$ for some $\sigma\in \gS_3$.
Note that for brevity, if $\sigma(1)=p$, $\sigma(2)=q$ and $\sigma(3)=r$, then we will often write $a_{pqr}$ for $a_\sigma$.
We use $\Tr$ to denote the  trace on $M_d(\C)$.

\begin{prop}[Mesyan \cite{Mes+}\footnote{We point out 
this result is proved in \cite{Mes+} for all fields with at least $d$ elements.}]\label{prop:kerTr}
$\Image f\supseteq M_d(\C)\cap\ker\Tr$.
\end{prop}

We present in Appendix \ref{app:proof} below an alternative proof of Proposition \ref{prop:kerTr}, for the special case where the field is $\C$, that uses clock-and-shift matrices.
The proposition allows us to focus on polynomials which are
not sums of commutators:

\begin{cor}[cf.~\protect{\cite[Corollary 15]{Mes+}}]
$\Image f=M_d(\C)\cap\ker\Tr$ if and only if  
$a_{123}+a_{231}+a_{312}=0= a_{132}+a_{213}+a_{321}$.
\end{cor}

\begin{proof}
By Proposition \ref{prop:kerTr},
$\Image f=M_d(\C)\cap\ker\Tr$ if and only if
 $\Image f\subseteq M_d(\C)\cap\ker\Tr$.
 Now
 $\Image f\subseteq M_d(\C)\cap\ker\Tr$ if and only if
$f$ is a sum of commutators \cite{BK09}.
This is in turn by Remark \ref{rem:cyceq} equivalent to
$a_{123}+a_{231}+a_{312}=0=a_{132}+a_{213}+a_{321}$.
\end{proof}

\begin{remark}
The analog of Proposition  \ref{prop:kerTr} for multilinear
polynomials $f$ of degree four is presented in \cite{BW13}. It
 implies
that the image of a sum of commutators $f$ equals
$M_d(\C)\cap\ker\Tr$ when $d\geq3$. 
The case $d=2$ has a few stray cases arising from the existence of central polynomials and polynomial identities of degree four.
\end{remark}

\subsection{One-wiggle families and the general strategy}\label{subsec:wiggle}

From now on we consider $f$ as in~\eqref{eq:f} and we will always
suppose
\begin{gather}
a_{123}+a_{132}+a_{213}+a_{231}+a_{312}+a_{321}=0, \label{eq:sum02co-0} \\
a_{123}+a_{231}+a_{312}\ne0. \label{eq:sumne0}
\end{gather}
We shall show that in various cases, $\Image f=M_d(\C)$.
Without loss of generality, we assume $a_{123}+a_{231}+a_{312}=1$.
Combined with~\eqref{eq:sum02co-0}, this implies
\begin{equation}\label{eq:asubs-0}
a_{312} = 1 - a_{123} - a_{231}, \qquad
a_{321} = -1 - a_{132} - a_{213}.
\end{equation}
This leaves the four coefficient values that determine $f$:
\begin{equation}\label{eq:acoefs-0}
a_{123},\quad a_{132},\quad a_{213},\quad a_{231}.
\end{equation}

\subsubsection{One-wiggle families}
We introduce some names we will use throughout this paper.
We will plug in as variables of $f$ unitaries from the set $\{u^p_q\mid p,q\in\{0,\ldots,d-1\}\}$;
we will fix two of the three variables of $f$ and  vary the third.
Then, since $f$ is multilinear, the linear span of all matrices so obtained will be contained in the $\Image f$.
Since $\Image f$ is closed under conjugation by invertible matrices, in order to establish that $\Image f=M_d(\C)$, it will suffice
to show that all Jordan canonical forms belong to $\Image f$.
Therefore, it will be enough to show,
\begin{equation}\label{eq:desiredupq}
\big\{u^0_q\mid q\in\{0,1,\ldots,d-1\}\big\}\cup\big\{u^1_q\mid q\in\{0,1,\ldots,d-1\}\big\}\subseteq\Image f.
\end{equation}
For example, in the case $d=3$, fixing the second and third variables of $f$, we find
\begin{align}
f(u^0_2,u^2_0,u^1_1)&=(-1 + \omega^2) u^0_0 \label{eq:owf1} \\
f(u^0_0,u^2_0,u^1_1)&=((-1 - a_{213} + a_{231})  + (1 + a_{213} - a_{231}) \omega^2) u^0_1 \\
f(u^0_1,u^2_0,u^1_1)&=((-1 - a_{213})  + (a_{213} + a_{231}) \omega + (1 - a_{231}) \omega^2) u^0_2 \\
f(u^1_2,u^2_0,u^1_1)&=((-1 + a_{123} - a_{132} - a_{213})  + (a_{132} + a_{213}) \omega + (1 - a_{123}) \omega^2) u^1_0 \\
f(u^1_0,u^2_0,u^1_1)&=((-1 + a_{123} - a_{132} + a_{231})  + a_{132} \omega + (1 - a_{123} - a_{231}) \omega^2) u^1_1 \\
f(u^1_1,u^2_0,u^1_1)&=
 \begin{aligned}[t]((-1 + a_{123} - a_{132} - a_{213})  & + (a_{132} + a_{231}) \omega  \\
          & + (1 - a_{123} + a_{213} - a_{231}) \omega^2) u^1_2
 \end{aligned} \label{eq:owf6} 
\end{align}

When we can fix two of the three variables in $f$ and ``wiggle'' (or vary) the other and by so doing we obtain constant multiples of all of the $u_q^p$ appearing in~\eqref{eq:desiredupq}, then we call this a {\bf one-wiggle family}.
Each one-wiggle family $H$ produces a list of coefficients of the various $u^p_q$ appearing on the right-hand-side.
In the case of the family \eqref{eq:owf1}-\eqref{eq:owf6}, using the identity $\omega^2=-1-\omega$, these coefficients are
\begin{alignat}{4}
-2-\omega&&&& \label{eq:owfcoef1} \\
-2-\omega&&&+(-2-\omega)a_{213}&+(2+\omega)a_{231} \label{eq:owfcoef2} \\
-2-\omega&&&+(-1+\omega)a_{213}&+(1+2\omega)a_{231} \\
-2-\omega&+(2+\omega)a_{123}&+(-1+\omega)a_{132}&+(-1+\omega)a_{213}& \\
-2-\omega&+(2+\omega)a_{123}&+(-1+\omega)a_{132}&&+(2+\omega)a_{231} \\
-2-\omega&+(2+\omega)a_{123}&+(-1+\omega)a_{132}&+(-2-\omega)a_{213}&+(1+2\omega)a_{231}. \label{eq:owfcoef6}
\end{alignat}
Thus, for any choice of values $a_{123},\,a_{132},\,a_{213},\,a_{231}$ making all six of the linear expressions \eqref{eq:owfcoef1}-\eqref{eq:owfcoef6}
nonzero, we will have $\Image f=M_3(\C)$.

Continuing with our example and the one-wiggle family $H$ described in \eqref{eq:owf1}-\eqref{eq:owf6},
the coefficient \eqref{eq:owfcoef1} is never zero, while
setting each of the other coefficients \eqref{eq:owfcoef2}-\eqref{eq:owfcoef6} to be zero
results in an affine subset of dimension $3$ in the $4$-dimensional parameter space for the coefficients~\eqref{eq:acoefs-0}.
Thus, all values of the coefficients~\eqref{eq:acoefs-0} except those in the union $\Omega(H)$
of five affine subsets of dimension $3$ will yield $\Image f=M_3(\C)$.

If we consider another one-wiggle family $H'$, and thereby obtain another union $\Omega(H')$ of affine subsets of the parameter space, outside of which
we have $\Image f=M_3(\C)$,
then taking these together, every choice of values of the coefficients~\eqref{eq:acoefs-0} outside of $\Omega(H)\cap\Omega(H')$ will yield $\Image f=M_3(\C)$.

\subsubsection{The general strategy}
Our goal will be to consider sufficiently many one-wiggle families $H_1,\ldots,H_k$
so that the intersection $\bigcap_1^k\Omega(H_j)$ of the corresponding unions of affine subsets is empty, which will then imply that $\Image f=M_d(\C)$
holds for all values of coefficients~\eqref{eq:acoefs-0}.

We will achieve this goal for other cases of $M_d(\C)$
from Theorem \ref{thm:main}, but,
as we will soon see, we cannot quite attain it in the case of $M_3(\C)$.
However, we will show $\bigcap_1^k\Omega(H_j)$ is small and treat remaining cases by other, ad hoc methods.

\subsection{$d=3$}\label{subsec:3.3}

In this subsection we 
use the one-wiggle families introduced above to
establish Theorem \ref{thm:main} for $d=3$.

Recall the Amitsur-Levitzki polynomial \cites{Row80,Pro73} of degree $n$ 
defined as
\beq\label{eq:al}
S_n(x_1,\ldots,x_n) = \sum_{\sigma\in \gS_n} \sign(\sigma) x_{\sigma(1)}
x_{\sigma(2)}\cdots x_{\sigma(n)}.
\eeq
It is the minimal polynomial identity in the sense that
$S_{2n}$ is (up to a scalar multiple) the only smallest degree minimal
polynomial vanishing on
$M_n(\C)$.

\begin{lemma}\label{lem:d=3allbutone}

Suppose $f$ is not a scalar multiple of the Amitsur-Levitzki polynomial $S_3(x_1,x_2,x_3)=x_1x_2x_3+x_2x_3x_1+x_3x_1x_2-x_1x_3x_2-x_3x_2x_1-x_2x_1x_3$.
Then $\Image f= M_3(\C)$.
\end{lemma}
\begin{proof}
We use $13$ one-wiggle families $H_1,\ldots,H_{13}$ and we show that $\bigcap_{j=1}^{13}\Omega(H_j)$ is the single point $(\frac13,-\frac13,-\frac13,\frac13)$,
which will prove the lemma.
The proof involved computation using Mathematica~\cite{Wol}.
We will describe the proof below and also a Mathematica notebook file with more details is available from the arXiv or the authors' websites. 

The first step of the proof is to find all one-wiggle families (eliminating those where a constant multiplying a $u^p_q$ in the image is obviously zero,
based on the identity~\eqref{eq:sum02co-0}).
Computation revealed that there are $144$ different one-wiggle families.

The second step is, from each one-wiggle family $H$, to describe the associated union $\Omega(H)$, of affine subsets of $\R^4$.
Each affine subspace is given as the solution set of a linear equation, thus, as a row vector of length $5$, whose
first four entries are the coefficients of the variables~\eqref{eq:acoefs-0} and whose last entry is the constant term.
Though for a given $H$, there are ostensibly six affine subsets whose union is $\Omega(H)$,
corresponding to setting the six quantities as in~\eqref{eq:owfcoef1}-\eqref{eq:owfcoef6} to be zero. In practice one of these is never zero
and the corresponding affine subset is empty;
thus it is discarded.
Thus, each $\Omega(H)$ is encoded as a set of five nonzero row vectors of length $5$, and we normalize so that the first nonzero entry is equal to $1$.
Though there were $144$ different one-wiggle families $H$, computation revealed that they yielded only $48$ different sets $\Omega(H)$.

The third step is to find one-wiggle families $H_1,\ldots,H_k$ so that
\begin{equation}\label{eq:ptint}
\bigcap_{j=1}^k\Omega(H_j)=\{\frac13,-\frac13,-\frac13,\frac13\}.
\end{equation}
This was accomplished, with $k=13$, as follows.
Here is how we keep track of affine subsets of $\R^4$ and their unions, and how we compute intersections of the these.
(This description is for working with general $d\times d$ matrices which we need later,
though in this lemma, of course, we have $d=3$;
this is especially relevant at points~\ref{it:AlgNum} and~\ref{it:shortcuts}.)
\begin{enumerate}[label={\rm(\alph*)}]
\item\label{it:rre} affine subsets of $\R^4$ are recorded as $\ell\times 5$ matrices in reduced row-echelon form;
\item the first four columns correspond to the coefficients of the variables~\eqref{eq:acoefs-0} and the last column is for the constant terms;
\item any rows of all zeros are dropped;
\item\label{it:001} if the matrix contains the row $(0,0,0,0,1)$ then the corresponding affine space is empty and this entire matrix is discarded;
\item an $\ell\times 5$ matrix corresponds to an affine subset of dimension $4-\ell$;
\item unions of affine subsets are recorded as sets of such row-reduced matrices;
\item\label{it:intersect} the intersection of two affine subsets is computed by forming the matrix consisting
of the rows of one stacked on top of the rows of the other,
finding the reduced row-echelon form of the matrix and discarding any zero rows; if the row $(0,0,0,0,1)$ is present then the intersection is empty and
may be discarded;
\item\label{it:AlgNum} the matrix computations are carried out in exact arithmetic using Mathematica's AlgebraicNumber facility,
in terms of the primitive $d$-th root of unity;
\item\label{it:ThetaPhi} the intersection of two sets $\Theta$ and $\Phi$
 that are unions of affine subsets is computed as the union of all pairwise intersections of the affine subsets, one from $\Theta$ and one from $\Phi$.
\end{enumerate}
The search for the one-wiggle families $H_1,\ldots,H_k$ so that the intersection
$\bigcap_{j=1}^k\Omega(H_j)$ is as small as possible, is carried out as follows:
\begin{enumerate}[label={\rm(\alph*)}]
\setcounter{enumi}{9}
\item\label{it:H1} an initial one-wiggle family $H_1$ is selected
\item if $H_1,\ldots,H_p$ have been selected and the intersection $\Phi=\bigcap_{j=1}^p\Omega(H_j)$ is recorded,
then all of the remaining unselected one-wiggle families $H$ are considered, and tested for the effect on the intersection;
\item ideally, $H_{p+1}$ is chosen so that the union $\Phi\cap\Omega(H_{p+1})$  of affine subsets of $\R^4$
is as good as possible, where ``good'' means
\begin{enumerate}[label={\rm(\arabic*)}]
\item the maximal dimension of the affine subsets is as small as possible;
\item in case of a tie for the maximal dimension, the number of  different affine subsets of maximal dimension in the union
is as small as possible;
\item in case of ties for the maximal dimension and the number of subspaces of maximal dimension, the number of different affine subsets
of dimension one less than maximal is as small as possible, etc.
\end{enumerate}
\item in case of ties in the above criteria, arbitrary selections are made
\item\label{it:shortcuts} in practice it can be computationally too expensive to test all of the intersections $\Phi\cap\Omega(H)$;
instead, we sometimes
\begin{enumerate}[label={\rm(\arabic*)}]
\item consider only $\Phi_{\max}\cap\Omega(H)$, where $\Phi_{\max}$ is the union of affine subsets of maximal dimension in $\Phi$;
\item test only some of the one-wiggle families $H$, by randomly selecting them and keeping track of the properties of $\Phi\cap\Omega(H)$;
if, after testing several, sufficient improvement in the goodness of $\Phi$ is observed, then the testing is terminated and $H_{p+1}$ is selected
to be the one that was tested that resulted in the most improvement.
\end{enumerate}
\end{enumerate}

As mentioned, we found $13$ one-wiggle families 
$H_1,\ldots,H_{13}$
so that~\eqref{eq:ptint} holds with $k=13$.
Without going into full details, these are summarized in Table~\ref{tab:d=3}.
The order of their numbering is the order in which they were selected by the algorithm described in parts \ref{it:H1}-\ref{it:shortcuts} above.
\begin{table}[h]
\caption{The one-wiggle families for $d=3$}
\label{tab:d=3}
\begin{tabular}{c|c||c|c||c|c}
name & form & name & form & name & form \\ \hline
$H_{1}$ & $f(u^{1}_{0},u^{2}_{2},u^{p}_{q})$ & $H_{2}$ & $f(u^{p}_{q},u^{1}_{0},u^{0}_{2})$ & $H_{3}$ & $f(u^{0}_{1},u^{p}_{q},u^{1}_{0})$ \\ \hline
$H_{4}$ & $f(u^{1}_{0},u^{p}_{q},u^{2}_{2})$ & $H_{5}$ & $f(u^{1}_{0},u^{p}_{q},u^{1}_{1})$ & $H_{6}$ & $f(u^{2}_{0},u^{1}_{1},u^{p}_{q})$ \\ \hline
$H_{7}$ & $f(u^{p}_{q},u^{2}_{1},u^{1}_{0})$ & $H_{8}$ & $f(u^{1}_{0},u^{1}_{1},u^{p}_{q})$ & $H_{9}$ & $f(u^{1}_{0},u^{p}_{q},u^{2}_{1})$ \\ \hline
$H_{10}$ & $f(u^{2}_{0},u^{2}_{1},u^{p}_{q})$ & $H_{11}$ & $f(u^{1}_{1},u^{1}_{0},u^{p}_{q})$ & $H_{12}$ & $f(u^{p}_{q},u^{2}_{0},u^{1}_{2})$ \\ \hline
$H_{13}$ & $f(u^{2}_{0},u^{1}_{2},u^{p}_{q})$
\end{tabular}
\end{table}

\vspace{-8mm}
\end{proof}

\begin{lem}\label{lem:al3}
For the Amitsur-Levitzki polynomial $f=S_3$, we have $\Image f=M_3(\C)$.
\end{lem}

\begin{proof}
It suffices to prove all possible $3\times3$ Jordan forms are in the
image of $S_3$. Letting
\[
N_3=\begin{pmatrix} 0 & 1 & 0 \\ 0&0&1\\ 0&0&0\end{pmatrix},
\]
it is easy to verify that
\[
\small
\begin{split}
\begin{pmatrix}
 \al  & 0 & 0 \\
 0 & \be & 0 \\
 0 & 0 & \ga
 \end{pmatrix} & =
S_3\left(
N_3, \begin{pmatrix}
 0 & 0 & 0 \\
 2 & 0 & 0 \\
 0 & 1 & 0
 \end{pmatrix},
 \frac16
 \begin{pmatrix}
  \frac{5 \al }{2}-2 \be+\ga & 0 & 0 \\
 0 & 2 \left(\al -2 \be+\ga\right) & 0 \\
 0 & 0 & \al -2 \be-2 \ga
 \end{pmatrix}
\right),
 \\
\begin{pmatrix}
 \al  & 1 & 0 \\
 0 & \al  & 0 \\
 0 & 0 & \be
  \end{pmatrix}
&=
 S_3\left(
N_3, \begin{pmatrix}
 -1 & 0 & 0 \\
 0 & 1 & 0 \\
 0 & 0 & 2
  \end{pmatrix},
  \frac13
  \begin{pmatrix}
   3 & 0 & 0 \\
 \al  & 0 & 0 \\
 0 & \be & 0
 \end{pmatrix}
\right), \\
\begin{pmatrix}
\al  & 1 & 0 \\
 0 & \al  & 1 \\
 0 & 0 & \al 
 \end{pmatrix}
&=
S_3\left(
N_3,
  \begin{pmatrix}
   1 & 0 & 0 \\
 1 & 0 & 0 \\
 0 & 1 & -1
 \end{pmatrix},
\frac12 \begin{pmatrix}
 \al -1 & 0 & 0 \\
 0 & -2 & 0 \\
 0 & 0 & -\al -1
  \end{pmatrix}
\right). \qedhere
\end{split}
\]
\end{proof}

Now the following is a result of Lemmas~\ref{lem:d=3allbutone} and~\ref{lem:al3}.
\begin{prop}\label{prop:deg3d=3}
For any multilinear polynomial $f$ with coefficients satisfying~\eqref{eq:sum02co-0} and~\eqref{eq:sumne0}, evaluated on triples
from $M_3(\C)$, we have $\Image f=M_3(\C)$.
\end{prop}

\subsection{$d<17$ odd}\label{subsec:17}

The cases of $d\times d$ matrices for $d$ odd, $5\le d\le15$, are proved similarly to the $d=3$ case,
though the method of intersecting affine subsets eliminated all values of the variables~\eqref{eq:acoefs-0} without exception.
\begin{prop}\label{prop:deg3dOdd5to13}
For any multilinear polynomial $f$ with coefficients satisfying~\eqref{eq:sum02co-0} and~\eqref{eq:sumne0}, evaluated on triples
from $M_d(\C)$, for $d$ odd and $5\le d\le 15$ and also for $d\in\{21,35\}$, we have $\Image f=M_d(\C)$.
\end{prop}
\begin{proof}
Using the method and algorithm described in the proof of Lemma~\ref{lem:d=3allbutone}, for each $d$ we found one-wiggle families $H_1,\ldots,H_{n}$
($n$ ranging from $6$ to $14$, depending on $d$)
such that $\bigcap_{j=1}^{n}\Omega(H_j)=\emptyset$, which provides a proof.
These one-wiggle families, numbered according to the order in which they were found by the algorithm, are described in Table~\ref{tab:d5to13}.
For those interested in more details, Mathematica 10.0 notebooks carrying out the algorithm in each case are available
from the arXiv or the authors' websites. 
\begin{table}[h]
\caption{The one-wiggle families for $d$ odd, $5\le d\le 15$ and for $d\in\{21,35\}$} \label{tab:d5to13}
\begin{tabular}{c||c|c||c|c||c|c}
$d$ & name & form & name & form & name & form \\ \hline
$d=5$
& $H_{1}$ & $f(u^{4}_{0},u^{p}_{q},u^{0}_{4})$ & $H_{2}$ & $f(u^{1}_{0},u^{4}_{4},u^{p}_{q})$ & $H_{3}$ & $f(u^{p}_{q},u^{3}_{0},u^{2}_{4})$ \\ \hline
& $H_{4}$ & $f(u^{2}_{0},u^{p}_{q},u^{4}_{1})$ & $H_{5}$ & $f(u^{4}_{0},u^{p}_{q},u^{3}_{4})$ & $H_{6}$ & $f(u^{3}_{0},u^{p}_{q},u^{1}_{3})$ \\ \hline
& $H_{7}$ & $f(u^{4}_{0},u^{p}_{q},u^{1}_{2})$ & $H_{8}$ & $f(u^{p}_{q},u^{3}_{0},u^{2}_{2})$ & $H_{9}$ & $f(u^{2}_{0},u^{1}_{4},u^{p}_{q})$ \\ \hline
& $H_{10}$ & $f(u^{2}_{0},u^{2}_{3},u^{p}_{q})$ & $H_{11}$ & $f(u^{1}_{0},u^{4}_{2},u^{p}_{q})$ & $H_{12}$ & $f(u^{3}_{0},u^{p}_{q},u^{2}_{3})$ \\ \hline
& $H_{13}$ & $f(u^{2}_{0},u^{1}_{3},u^{p}_{q})$ \\ \hline\hline
$d=7$
& $H_{1}$ & $f(u^{3}_{0},u^{2}_{1},u^{p}_{q})$ & $H_{2}$ & $f(u^{1}_{0},u^{p}_{q},u^{1}_{6})$ & $H_{3}$ & $f(u^{p}_{q},u^{5}_{1},u^{6}_{0})$ \\ \hline
& $H_{4}$ & $f(u^{6}_{0},u^{p}_{q},u^{0}_{3})$ & $H_{5}$ & $f(u^{1}_{0},u^{2}_{2},u^{p}_{q})$ & $H_{6}$ & $f(u^{p}_{q},u^{3}_{0},u^{3}_{1})$ \\ \hline
& $H_{7}$ & $f(u^{3}_{0},u^{3}_{5},u^{p}_{q})$ & $H_{8}$ & $f(u^{p}_{q},u^{5}_{0},u^{2}_{5})$ & $H_{9}$ & $f(u^{4}_{0},u^{p}_{q},u^{3}_{2})$ \\ \hline
& $H_{10}$ & $f(u^{4}_{0},u^{3}_{6},u^{p}_{q})$& $H_{11}$ & $f(u^{p}_{q},u^{3}_{0},u^{4}_{2})$ & $H_{12}$ & $f(u^{2}_{0},u^{p}_{q},u^{5}_{6})$ \\ \hline
& $H_{13}$ & $f(u^{p}_{q},u^{5}_{0},u^{2}_{1})$ & $H_{14}$ & $f(u^{1}_{0},u^{6}_{4},u^{p}_{q})$ \\ \hline\hline
$d=9$
& $H_{1}$ & $f(u^{6}_{1},u^{6}_{2},u^{p}_{q})$ & $H_{2}$ & $f(u^{3}_{0},u^{0}_{5},u^{p}_{q})$ & $H_{3}$ & $f(u^{0}_{1},u^{3}_{0},u^{p}_{q})$ \\ \hline
& $H_{4}$ & $f(u^{p}_{q},u^{6}_{1},u^{3}_{0})$ & $H_{5}$ & $f(u^{3}_{8},u^{p}_{q},u^{6}_{8})$ & $H_{6}$ & $f(u^{3}_{8},u^{6}_{8},u^{p}_{q})$ \\ \hline\hline
$d=11$
& $H_{1}$ & $f(u^{2}_{0},u^{p}_{q},u^{8}_{1})$ & $H_{2}$ & $f(u^{p}_{q},u^{1}_{0},u^{6}_{3})$ & $H_{3}$ & $f(u^{p}_{q},u^{6}_{0},u^{5}_{10})$ \\ \hline
& $H_{4}$ & $f(u^{10}_{0},u^{6}_{5},u^{p}_{q})$ & $H_{5}$ & $f(u^{1}_{0},u^{p}_{q},u^{10}_{10})$ & $H_{6}$ & $f(u^{9}_{0},u^{9}_{5},u^{p}_{q})$ \\ \hline
& $H_{7}$ & $f(u^{1}_{0},u^{7}_{3},u^{p}_{q})$ & $H_{8}$ & $f(u^{5}_{0},u^{6}_{6},u^{p}_{q})$ & $H_{9}$ & $f(u^{8}_{0},u^{p}_{q},u^{5}_{1})$ \\ \hline
& $H_{10}$ & $f(u^{8}_{0},u^{3}_{8},u^{p}_{q})$ & $H_{11}$ & $f(u^{10}_{0},u^{p}_{q},u^{1}_{1})$ & $H_{12}$ & $f(u^{p}_{q},u^{6}_{0},u^{5}_{4})$ \\ \hline\hline
$d=13$
& $H_{1}$ & $f(u^{p}_{q},u^{1}_{0},u^{11}_{11})$ & $H_{2}$ & $f(u^{9}_{0},u^{p}_{q},u^{4}_{6})$ & $H_{3}$ & $f(u^{p}_{q},u^{10}_{0},u^{3}_{1})$ \\ \hline
& $H_{4}$ & $f(u^{2}_{0},u^{9}_{1},u^{p}_{q})$ & $H_{5}$ & $f(u^{7}_{0},u^{3}_{8},u^{p}_{q})$ & $H_{6}$ & $f(u^{p}_{q},u^{7}_{0},u^{9}_{5})$ \\ \hline
& $H_{7}$ & $f(u^{p}_{q},u^{9}_{0},u^{8}_{2})$ & $H_{8}$ & $f(u^{9}_{0},u^{4}_{7},u^{p}_{q})$ & $H_{9}$ & $f(u^{p}_{q},u^{9}_{0},u^{4}_{3})$ \\ \hline
& $H_{10}$ & $f(u^{1}_{0},u^{12}_{3},u^{p}_{q})$ & $H_{11}$ & $f(u^{5}_{0},u^{p}_{q},u^{8}_{5})$
\\ \hline\hline
$d=15$
& $H_{1}$ & $f(u^{0}_{13},u^{5}_{8},u^{p}_{q})$ & $H_{2}$ & $f(u^{10}_{10},u^{p}_{q},u^{0}_{8})$ & $H_{3}$ & $f(u^{p}_{q},u^{0}_{8},u^{5}_{2})$ \\ \hline
& $H_{4}$ & $f(u^{10}_{4},u^{5}_{0},u^{p}_{q})$ & $H_{5}$ & $f(u^{p}_{q},u^{6}_{3},u^{12}_{12})$ & $H_{6}$ & $f(u^{9}_{1},u^{p}_{q},u^{6}_{10})$
\\ \hline\hline
$d=21$
& $H_{1}$ & $f(u^{p}_{q},u^{14}_{12},u^{7}_{13})$ & $H_{2}$ & $f(u^{14}_{7},u^{14}_{15},u^{p}_{q})$ & $H_{3}$ & $f(u^{14}_{7},u^{0}_{11},u^{p}_{q})$ \\ \hline
& $H_{4}$ & $f(u^{0}_{10},u^{7}_{1},u^{p}_{q})$ & $H_{5}$ & $f(u^{14}_{16},u^{p}_{q},u^{14}_{17})$ & $H_{6}$ & $f(u^{2}_{8},u^{p}_{q},u^{19}_{9})$ \\ \hline
& $H_{7}$ & $f(u^{15}_{6},u^{6}_{2},u^{p}_{q})$ 
\\ \hline\hline
$d=35$
& $H_{1}$ & $f(u^{p}_{q},u^{7}_{6},u^{28}_{0})$ & $H_{2}$ & $f(u^{p}_{q},u^{14}_{29},u^{28}_{31})$ & $H_{3}$ & $f(u^{p}_{q},u^{21`}_{17},u^{0}_{28})$ \\ \hline
& $H_{4}$ & $f(u^{0}_{27},u^{p}_{q},u^{21}_{28})$ & $H_{5}$ & $f(u^{14}_{31},u^{p}_{q},u^{0}_{29})$ & $H_{6}$ & $f(u^{29}_{33},u^{p}_{q},u^{6}_{6})$ \\ \hline
& $H_{7}$ & $f(u^{22}_{20},u^{13}_{20},u^{p}_{q})$ 
\\ \hline
\end{tabular}
\end{table}
%
%
%

\end{proof}

\smallskip

\subsection{$d$ even but not a multiple of 4}\label{subsec:even}

\begin{prop}\label{prop:d=2cOdd}
For any multilinear polynomial $f$ with coefficients satisfying~\eqref{eq:sum02co-0} and~\eqref{eq:sumne0}, evaluated on triples
from $M_d(\C)$, for $d\ge6$ even but not a multiple of $4$,
we have $\Image f=M_d(\C)$.
\end{prop}
\begin{proof}
We write $d=2c$ with $c$ odd.
We consider one-wiggle families.
Since $\omega$ is a primitive $d$-th root of unity and $d=2c$ with $c$ odd
we have $\omega^c=-1$ and $\omega^{c^2}=-1$.
The basic idea of the proof is similar to that of the odd cases treated in Lemma~\ref{lem:d=3allbutone} and Proposition~\ref{prop:deg3dOdd5to13},
but with simplification provided by some happy coincidences in the coefficients for certain one-wiggle families.
For example, we get
\[
f(u^c_0,u^c_c,u^0_q)=\al\,u^0_{c+q}
\]
where
\[
\al=\begin{cases}
-2 - 2 a_{132} + 2 a_{231},&q\text{ even}, \\
-2,&q\text{ odd}
\end{cases}
\]
and
\[
f(u^c_0,u^c_c,u^1_q)=\be\,u^1_{c+q}
\]
where
\[
\be=\begin{cases}
2 - 2 a_{123} + 2 a_{132} + 2 a_{213},&q\text{ even}, \\
2 - 2 a_{123} +  2 a_{213} - 2 a_{231},&q\text{ odd},
\end{cases}
\]
where we have made the substitutions~\eqref{eq:asubs-0}.
Thus, for any choice of coefficients~\eqref{eq:acoefs-0} making the three expressions
\begin{equation}\label{eq:exprfam1}
-2 - 2 a_{132} + 2 a_{231},\qquad
2 - 2 a_{123} + 2 a_{132} + 2 a_{213},\qquad
2 - 2 a_{123} +  2 a_{213} - 2 a_{231}
\end{equation}
nonzero, $\Image f=M_d(\C)$ is proved.
We will need eight such one-wiggle families, and these, named $H_1,\ldots,H_8$, are described in Table~\ref{tab:2co}.
\begin{center}
\begin{longtable}{c|l|c|l}
\caption{Values of $f$ used in the case $d=2c$ with $c$ odd} \\
\label{tab:2co}
fam. & $f$ evaluation & yields $u^*_*$ & times this quantity \\ \hline\hline
\endfirsthead
\multicolumn{4}{c}%
{\tablename\ \thetable\ -- \textit{Continued from previous page}} \\
\hline
fam. & $f$ evaluation & yields $u^*_*$ & times this quantity \\ \hline
\endhead
\hline \multicolumn{4}{r}{\textit{Continued on next page}} \\
\endfoot
\hline
\endlastfoot
\rule{0ex}{2.5ex}
$H_1$ & $f(u^c_0,u^c_c,u^0_q)$, $q$ even & $u^0_{c+q}$ & $-2 - 2 a_{132} + 2 a_{231}$ \\[0.5ex] \hline
\rule{0ex}{2.5ex}
& $f(u^c_0,u^c_c,u^0_q)$, $q$ odd & $u^0_{c+q}$ & $-2$ \\[0.5ex] \hline
\rule{0ex}{2.5ex}
& $f(u^c_0,u^c_c,u^1_q)$, $q$ even  & $u^1_{c+q}$ & $2 - 2 a_{123} + 2 a_{132} + 2 a_{213}$ \\[0.5ex] \hline
\rule{0ex}{2.5ex}
& $f(u^c_0,u^c_c,u^1_q)$, $q$ odd  & $u^1_{c+q}$ & $2 - 2 a_{123} +  2 a_{213} - 2 a_{231}$  \\[0.5ex] \hline\hline
\rule{0ex}{2.5ex}
$H_2$ & $f(u^c_1,u^c_{c+1},u^0_q)$, $q$ even & $u^0_{2+c+q}$ & $2 + 2 a_{132} - 2 a_{231}$ \\[0.5ex] \hline
\rule{0ex}{2.5ex}
& $f(u^c_1,u^c_{c+1},u^0_q)$, $q$ odd & $u^0_{2+c+q}$ & $2$ \\[0.5ex] \hline
\rule{0ex}{2.5ex}
& $f(u^c_1,u^c_{c+1},u^1_q)$, $q$ even & $u^1_{2+c+q}$ &
 $-2 \omega^2
  + (1 + \omega^2) a_{123}
  - (\omega + \omega^2) a_{132}$  \\[0.5ex]
\rule{0ex}{2.5ex}
&&& \hspace*{1em}
 $-\, (1 + \omega^2) a_{213}  
  - (\omega - \omega^2) a_{231} $ \\[0.5ex] \hline
\rule{0ex}{2.5ex}
& $f(u^c_1,u^c_{c+1},u^1_q)$, $q$ odd & $u^1_{2+c+q}$ &
 $-  2 \omega^2
   +(1 + \omega^2) a_{123} 
   + (\omega - \omega^2) a_{132} $  \\[0.5ex]
\rule{0ex}{2.5ex}
&&& \hspace*{1em}
 $-\,(1 + \omega^2) a_{213}  
   +  (\omega + \omega^2) a_{231}$
 \\[0.5ex] \hline\hline
\rule{0ex}{2.5ex}
$H_3$ & $f(u^0_q,u^c_0,u^c_c)$, $q$ even & $u^0_{c+q}$ & $-2 a_{123} - 2 a_{213} - 2 a_{231}$ \\[0.5ex] \hline
\rule{0ex}{2.5ex}
& $f(u^0_q,u^c_0,u^c_c)$, $q$ odd & $u^0_{c+q}$ & $-2$ \\[0.5ex] \hline
\rule{0ex}{2.5ex}
& $f(u^1_q,u^c_0,u^c_c)$, $q$ even & $u^1_{c+q}$ & $-2 a_{132} - 2 a_{231}$ \\[0.5ex] \hline
\rule{0ex}{2.5ex}
& $f(u^1_q,u^c_0,u^c_c)$, $q$ odd & $u^1_{c+q}$ & $-2 + 2 a_{123} - 2 a_{132} - 2 a_{213}$ \\[0.5ex] \hline\hline
\rule{0ex}{2.5ex}
$H_4$ & $f(u^c_0,u^0_q,u^c_c)$, $q$ even & $u^0_{c+q}$ & $-2 a_{123} - 2 a_{132} - 2 a_{213}$ \\[0.5ex] \hline
\rule{0ex}{2.5ex}
& $f(u^c_0,u^0_q,u^c_c)$, $q$ odd & $u^0_{c+q}$ & $2$ \\[0.5ex] \hline
\rule{0ex}{2.5ex}
& $f(u^c_0,u^1_q,u^c_c)$, $q$ even & $u^1_{c+q}$ & $-2 a_{132} - 2 a_{231}$ \\[0.5ex] \hline
\rule{0ex}{2.5ex}
& $f(u^c_0,u^1_q,u^c_c)$, $q$ odd & $u^1_{c+q}$ & $2 - 2 a_{123} + 2 a_{213} - 2 a_{231}$ \\[0.5ex] \hline\hline
\rule{0ex}{2.5ex}
$H_5$ & $f(u^0_q,u^c_{c-1},u^c_c)$, $q$ even & $u^0_{q-1}$ & $-2 a_{123} - 2 a_{213} - 2 a_{231}$ \\[0.5ex] \hline
\rule{0ex}{2.5ex}
& $f(u^0_q,u^c_{c-1},u^c_c)$, $q$ odd & $u^0_{q-1}$ & $-2$ \\[0.5ex] \hline
\rule{0ex}{2.5ex}
& $f(u^1_q,u^c_{c-1},u^c_c)$, $q$ even & $u^1_{q-1}$ &
 $-(1+ \omega^{-1}) - (1 - \omega^{-1}) a_{132} $  \\[0.5ex]
\rule{0ex}{2.5ex}
&&& \hspace*{1em}
 $-(1 - \omega^{-1}) a_{231}$ \\[0.5ex] \hline
\rule{0ex}{2.5ex}
& $f(u^1_q,u^c_{c-1},u^c_c)$, $q$ odd & $u^1_{q-1}$ &
 $-(1 - \omega^{-1})
  - 2 \omega^{-1} a_{123}$  \\[0.5ex]
\rule{0ex}{2.5ex}
&&& \hspace*{1em}
 $-\,(1 - \omega^{-1}) a_{132}
   - 2 a_{213}$  \\[0.5ex]
\rule{0ex}{2.5ex}
&&& \hspace*{1em}
 $-\, (1 + \omega^{-1}) a_{231}$ \\[0.5ex] \hline\hline 
\rule{0ex}{2.5ex}
$H_6$ & $f(u^c_1,u^0_q,u^c_0)$, $q$ even & $u^0_{q+1}$ & $2 a_{123} + 2 a_{132} + 2 a_{213}$ \\[0.5ex] \hline
\rule{0ex}{2.5ex}
& $f(u^c_1,u^0_q,u^c_0)$, $q$ odd & $u^0_{q+1}$ & $-2$ \\[0.5ex] \hline
\rule{0ex}{2.5ex}
& $f(u^c_1,u^1_q,u^c_0)$, $q$ even & $u^1_{q+1}$ &
 $-1+ \omega  + 2 a_{123}+  (1 + \omega) a_{132}$  \\[0.5ex]
\rule{0ex}{2.5ex}
&&& \hspace*{1em}
 $+  2 \omega a_{213}  +  (1 - \omega) a_{231}$ \\[0.5ex] \hline
\rule{0ex}{2.5ex}
& $f(u^c_1,u^1_q,u^c_0)$, $q$ odd & $u^1_{q+1}$ &
 $-(1+ \omega) + (1 - \omega) a_{132} + (1 - \omega) a_{231} $ \\[0.5ex] \hline\hline
\rule{0ex}{2.5ex}
$H_7$ & $f(u^c_q,u^0_1,u^c_c)$, $q$ even & $u^0_{q+c+1}$ & $2$ \\[0.5ex] \hline
\rule{0ex}{2.5ex}
& $f(u^c_q,u^0_1,u^c_c)$, $q$ odd & $u^0_{q+c+1}$ & $-2 + 2 a_{123} - 2 a_{132} - 2 a_{213}$ \\[0.5ex] \hline
\rule{0ex}{2.5ex}
& $f(u^{c+1}_q,u^0_1,u^c_c)$, $q$ even & $u^1_{q+c+1}$ &
 $1+ \omega
 -  2 a_{123} \omega 
 + (1 + \omega) a_{132}$  \\[0.5ex]
\rule{0ex}{2.5ex}
&&& \hspace*{1em}
 $+\,2 a_{213}
 + (1 - \omega)  a_{231}$ \\[0.5ex] \hline
\rule{0ex}{2.5ex}
& $f(u^{c+1}_q,u^0_1,u^c_c)$, $q$ odd & $u^1_{q+c+1}$ &
 $-(1 + \omega)
-(1 - \omega) a_{132}
-(1 - \omega) a_{231}$  \\[0.5ex]\hline\hline
\rule{0ex}{2.5ex}
$H_8$ & $f(u^0_{-1},u^c_q,u^c_0)$, $q$ even & $u^0_{q-1}$ & $-2 + 2 a_{123} - 2 a_{213} + 2 a_{231}$ \\[0.5ex] \hline
\rule{0ex}{2.5ex}
& $f(u^0_{-1},u^c_q,u^c_0)$, $q$ odd & $u^0_{q-1}$ & $2$ \\[0.5ex] \hline
\rule{0ex}{2.5ex}
& $f(u^0_{-1},u^{c+1}_q,u^c_0)$, $q$ even & $u^1_{q-1}$ &
 $-(1 + \omega^{-1})
 + 2 a_{123} 
 + (1 - \omega^{-1}) a_{132}$  \\[0.5ex]
\rule{0ex}{2.5ex}
&&& \hspace*{1em}
 $ -\, 2 \omega^{-1}  a_{213}
+  (1 + \omega^{-1}) a_{231}$ \\[0.5ex] \hline
\rule{0ex}{2.5ex}
& $f(u^0_{-1},u^{c+1}_q,u^c_0)$, $q$ odd & $u^1_{q-1}$ &
 $1 + \omega^{-1}
 - (1 - \omega^{-1}) a_{132}$  \\[0.5ex]
\rule{0ex}{2.5ex}
&&& \hspace*{1em}
 $ -\, (1 - \omega^{-1}) a_{231}$ \\[0.5ex] \hline\hline
\end{longtable}
\end{center}

\vspace{-12mm}

\noindent
As in the previous treatments, for each family $H_i$ we have a union $\Omega(H_i)$ of three affine subsets of $\R^4$, namely, the
subsets described by setting the expressions in
the last column of Table~\ref{tab:2co} equal to zero.
For example, in the case of the family $H_1$, these are the expressions~\eqref{eq:exprfam1}.
(In the case of the third and fourth rows of families $H_5$ and $H_8$, we first multiply
the entire quantity by $\omega$ to get nonnegative powers of $\omega$.)
We show that the intersection $\bigcap_{i=1}^8\Omega(H_i)$ is empty, which proves $\Image f=M_d(\C)$.
The technique is similar to that described in items~\ref{it:rre}-\ref{it:ThetaPhi} except that when performing Gaussian elimination
on matrices representing affine subsets, instead of working over the field of algebraic numbers generated by a specific primitive $d$-th
root of unity, we work over the ring of polynomials in the indeterminant $\omega$, where $\omega$ represents a primitive $d$-th root of unity
for unknown $d$,
and we are careful never to divide by a polynomial of positive degree.
With this it is only possible, in general, to reduce to a matrix in row-echelon form where the leading nonzero entry of each row
is a monic polynomial and all entries above it have strictly lower degree in $\omega$.
In particular, item~\ref{it:AlgNum} no longer applies
and items~\ref{it:001} and~\ref{it:intersect} are modified;
note that anytime
a row $(0,0,0,0,p)$ appears in a matrix representing an affine subset, where $p$ is a nonzero polynomial in $\omega$, then the affine subset
is empty provided $d$ is such that $p$ applied to a primitive $d$-th root of unity is nonzero.
Thus, items~\ref{it:001} and~\ref{it:intersect} are modified so that whenever such a row $(0,0,0,0,p)$ appears, then the matrix is discarded,
but a note is made that it is assumed $p(\omega)\ne0$.

In the end, we obtained that the intersection $\bigcap_{i=1}^8\Omega(H_i)$ is empty, provided that all in a list of polynomials in $\omega$
are nonvanishing;  the irreducible factors of the polynomials in this list are:
\[
\omega -1,\;\omega ,\;\omega +1,\;\omega ^2+1,\;\omega ^2-6 \omega +1,\;3 \omega ^2-2 \omega +3.
\]
It is easy to see that none of these polynomials vanish on the primitive $d$-th root of unity when $d\ge3$.
This proves the proposition.
The details of the calculations can be found in a Mathematica 10.0 notebook available from the arXiv or the authors' websites. 
\end{proof}

\eject

\subsection{$d$ multiple of 4}\label{subsec:4}

\begin{prop}\label{prop:d=2cEven}
For any multilinear polynomial $f$
as in \eqref{eq:f}
 with coefficients satisfying~\eqref{eq:sum02co-0} and~\eqref{eq:sumne0}, evaluated on triples
from $M_d(\C)$, for $d\ge4$ a multiple of $4$,
we have $\Image f=M_d(\C)$.
\end{prop}
\begin{proof}
We write $d=2c$ with $c$ even, $c\ge2$.
We proceed as in the proof of Proposition~\ref{prop:d=2cOdd}.
Computations are summarized in Table~\ref{tab:2ce}, below.
\begin{center}
\begin{longtable}{c|l|c|l}
\caption{Values of $f$ used in the case $d=2c$ with $c$ even} \\
\label{tab:2ce}
fam. & $f$ evaluation & yields $u^*_*$ & times this quantity \\ \hline\hline
\endfirsthead
\multicolumn{4}{c}%
{\tablename\ \thetable\ -- \textit{Continued from previous page}} \\
\hline
fam. & $f$ evaluation & yields $u^*_*$ & times this quantity \\ \hline
\endhead
\hline \multicolumn{4}{r}{\textit{Continued on next page}} \\
\endfoot
\hline
\endlastfoot
\rule{0ex}{2.5ex}
$H_1$ & $f(u^c_1,u^c_c,u^0_q)$, $q$ even & $u^0_{c+q+1}$ & $2 + 2 a_{132} - 2 a_{231}$ \\[0.5ex] \hline
\rule{0ex}{2.5ex}
& $f(u^c_1,u^c_c,u^0_q)$, $q$ odd & $u^0_{c+q+1}$ & $2$ \\[0.5ex] \hline
\rule{0ex}{2.5ex}
& $f(u^c_1,u^c_c,u^1_q)$, $q$ even  & $u^1_{c+q+1}$ &
$- 2 \omega + (1 + \omega)a_{123} - (1 + \omega)a_{132}$ \\[0.5ex]
\rule{0ex}{2.5ex}
&&& \hspace*{1em}
$-\,(1 + \omega) a_{213}$ \\[0.5ex] \hline
\rule{0ex}{2.5ex}
& $f(u^c_1,u^c_c,u^1_q)$, $q$ odd  & $u^1_{c+q+1}$ &
$- 2 \omega +(1 + \omega) a_{123}  + (1 - \omega) a_{132}  $  \\[0.5ex]
\rule{0ex}{2.5ex}
&&& \hspace*{1em}
 $ -\, (1 + \omega) a_{213} +  2 \omega a_{231} $ \\[0.5ex] \hline\hline
\rule{0ex}{2.5ex}
$H_2$ & $f(u^c_1,u^0_1,u^c_q)$, $q$ even & $u^0_{q+2}$ & $-2$ \\[0.5ex] \hline
\rule{0ex}{2.5ex}
& $f(u^c_1,u^0_1,u^c_q)$, $q$ odd & $u^0_{q+2}$ & $-2 + 2 a_{123} - 2 a_{132} - 2 a_{213}$ \\[0.5ex] \hline
\rule{0ex}{2.5ex}
& $f(u^c_1,u^0_1,u^{c+1}_q)$, $q$ even & $u^1_{q+2}$ &
 $- 2 \omega^2 - (1 - \omega^2) a_{123}   + (\omega - \omega^2)a_{132} $  \\[0.5ex]
\rule{0ex}{2.5ex}
&&& \hspace*{1em}
 $+\, (1 - \omega^2) a_{213} -  (\omega - \omega^2) a_{231}$ \\[0.5ex] \hline
\rule{0ex}{2.5ex}
& $f(u^c_1,u^0_1,u^{c+1}_q))$, $q$ odd & $u^1_{q+2}$ &
 $-2 \omega^2  + (1 + \omega^2) a_{123} - (\omega + \omega^2) a_{132}$  \\[0.5ex]
\rule{0ex}{2.5ex}
&&& \hspace*{1em}
 $ -\, (1 + \omega^2) a_{213}  -(\omega - \omega^2) a_{231} $
 \\[0.5ex] \hline\hline
\rule{0ex}{2.5ex}
$H_3$ & $f(u^c_q,u^c_0,u^0_1)$, $q$ even & $u^0_{q+1}$ & $-2 a_{132} - 2 a_{231}$ \\[0.5ex] \hline
\rule{0ex}{2.5ex}
& $f(u^c_q,u^c_0,u^0_1)$, $q$ odd & $u^0_{q+1}$ & $2$ \\[0.5ex] \hline
\rule{0ex}{2.5ex}
& $f(u^{c+1}_q,u^c_0,u^0_1)$, $q$ even & $u^1_{q+1}$ & 
$-(1 - \omega) a_{123}   -(1 + \omega) a_{132}$ \\[0.5ex]
\rule{0ex}{2.5ex}
&&& \hspace*{1em}
$ - (1 - \omega) a_{213} -2 a_{231}$ \\[0.5ex] \hline
\rule{0ex}{2.5ex}
& $f(u^{c+1}_q,u^c_0,u^0_1)$, $q$ odd & $u^1_{q+1}$ &
$2 - (1 - \omega)  a_{123} +  (1 - \omega)a_{132}$  \\[0.5ex]
\rule{0ex}{2.5ex}
&&& \hspace*{1em}
$+\, (1 - \omega) a_{213}$ \\[0.5ex] \hline\hline
\rule{0ex}{2.5ex}
$H_4$ & $f(u^c_q,u^0_1,u^c_0)$, $q$ even & $u^0_{q+1}$ & $-2 a_{132} - 2 a_{231}$ \\[0.5ex] \hline
\rule{0ex}{2.5ex}
& $f(u^c_q,u^0_1,u^c_0)$, $q$ odd & $u^0_{q+1}$ & $2$ \\[0.5ex] \hline
\rule{0ex}{2.5ex}
& $f(u^{c+1}_q,u^0_1,u^c_0)$, $q$ even & $u^1_{q+1}$ &
$1+ \omega -  2 \omega a_{123}+ (1 + \omega) a_{132}$ \\[0.5ex]
\rule{0ex}{2.5ex}
&&& \hspace*{1em}
$+\, 2 a_{213} + (1 - \omega) a_{231}$ \\[0.5ex] \hline
\rule{0ex}{2.5ex}
& $f(u^{c+1}_q,u^0_1,u^c_0)$, $q$ odd & $u^1_{q+1}$ &
$-(1+ \omega) - (1 - \omega) a_{132} - (1 - \omega) a_{231}  $ \\[0.5ex] \hline\hline
\rule{0ex}{2.5ex}
$H_5$ & $f(u^0_q,u^c_0,u^c_1)$, $q$ even & $u^0_{q+1}$ & $-2 a_{123} - 2 a_{213} - 2 a_{231}$ \\[0.5ex] \hline
\rule{0ex}{2.5ex}
& $f(u^0_q,u^c_0,u^c_1)$, $q$ odd & $u^0_{q+1}$ & $-2$ \\[0.5ex] \hline
\rule{0ex}{2.5ex}
& $f(u^1_q,u^c_0,u^c_1)$, $q$ even & $u^1_{q+1}$ &
 $-(1 + \omega)a_{123} -(1 - \omega) a_{132}$ \\[0.5ex]
\rule{0ex}{2.5ex}
&&& \hspace*{1em}
$-\,(1 + \omega) a_{213}-2 a_{231} $ \\[0.5ex] \hline
\rule{0ex}{2.5ex}
& $f(u^1_q,u^c_0,u^c_1)$, $q$ odd & $u^1_{q+1}$ &
$-2 + (1 - \omega) a_{123} -(1 - \omega) a_{132} $ \\[0.5ex]
\rule{0ex}{2.5ex}
&&& \hspace*{1em}
$ -\,(1 - \omega)   a_{213}$ \\[0.5ex] \hline\hline
\rule{0ex}{2.5ex}
$H_6$ & $f(u^c_0,u^0_q,u^c_1)$, $q$ even & $u^0_{q+1}$ & $-2 a_{123} - 2 a_{132} - 2 a_{213}$ \\[0.5ex] \hline
\rule{0ex}{2.5ex}
& $f(u^c_0,u^0_q,u^c_1)$, $q$ odd & $u^0_{q+1}$ & $2$ \\[0.5ex] \hline
\rule{0ex}{2.5ex}
& $f(u^c_0,u^1_q,u^c_1)$, $q$ even & $u^1_{q+1}$ &
 $ -(1 + \omega) a_{123} -2 a_{132} -(1 + \omega) a_{213}$ \\[0.5ex]
\rule{0ex}{2.5ex}
&&& \hspace*{1em}
$-\,(1 - \omega) a_{231} $  \\[0.5ex] \hline
\rule{0ex}{2.5ex}
& $f(u^c_0,u^0_q,u^c_1)$, $q$ odd & $u^1_{q+1}$ &
 $2 - (1 - \omega) a_{123} + (1 - \omega) a_{213} $  \\[0.5ex]
\rule{0ex}{2.5ex}
&&& \hspace*{1em}
$-\, (1 - \omega)a_{231}$ \\[0.5ex]\hline
\end{longtable}
\end{center}

\vspace{-5mm}

\noindent
Proceeding as described in the proof of Proposition~\ref{prop:d=2cOdd}, we show that the intersection $\bigcap_{i=1}^6\Omega(H_i)$
is empty provided none of the polynomials
\begin{multline*}
\omega -3,\;\omega -1,\;\omega ,\;\omega +1,\;3 \omega -1,\;\omega ^2+1,\;\omega ^2-4 \omega +1,\\
\omega ^2+4 \omega +1,\;3 \omega ^2-2 \omega
 +3,\;\omega ^4-2 \omega ^3+10 \omega ^2-2 \omega +1
 \end{multline*}
 vanish and it is easy to show that none of these vanish when $\omega$ is a primitive $d$-th root of unity for $d\ge5$.
However, in case $d=4$, namely $\omega=i$, the polynomial $w^2+1$ does, of course, vanish.

It remains to consider the case $\omega=i$.
For this case, we need to add one more one-wiggle family, $H_7$, presented in Table~\ref{tab:2ce7th}.
\begin{table}[!htbp]
\caption{The extra one-wiggle family used in the case $d=4$}
\begin{tabular}{c|l|c|l}
\label{tab:2ce7th}
fam. & $f$ evaluation & yields $u^*_*$ & times this quantity \\ \hline\hline
$H_7$ & $f(u^c_0,u^0_1,u^c_q)$, $q$ even & $u^0_{q+1}$ & $2 - 2 a_{123} + 2 a_{132} + 2 a_{213}$ \\[0.5ex] \hline
\rule{0ex}{2.5ex}
& $f(u^c_0,u^0_1,u^c_q)$, $q$ odd & $u^0_{q+1}$ & $2$ \\[0.5ex] \hline
\rule{0ex}{2.5ex}
& $f(u^c_0,u^0_1,u^{c+1}_q)$, $q$ even & $u^1_{q+1}$ &
 $2 \omega - (1 + \omega) a_{123} + 2 \omega a_{132}$  \\[0.5ex]
\rule{0ex}{2.5ex}
&&& \hspace*{1em}
 $ +\, (1 + \omega) a_{213} + (1 - \omega) a_{231}$ \\[0.5ex] \hline
\rule{0ex}{2.5ex}
& $f(u^c_0,u^0_1,u^{c+1}_q)$, $q$ odd & $u^1_{q+1}$ &
$2 \omega + (1 - \omega) a_{123} - (1 - \omega) a_{213} $  \\[0.5ex]
\rule{0ex}{2.5ex}
&&& \hspace*{1em}
$ +\, (1 - \omega) a_{231}$  \\[0.5ex] \hline\hline
\end{tabular}
\end{table}
Substituting $\omega=i$ and executing the algorithm as described in the proof of Lemma~\ref{lem:d=3allbutone},
we easily show $\bigcap_{i=1}^7\Omega(H_i)=\varnothing$ in this case, which finishes the proof.
\end{proof}


\appendix

\section{Alternative Proof of Proposition \ref{prop:kerTr}}\label{app:proof}

In this appendix we present a short and self-contained proof
of Proposition \ref{prop:kerTr}. We start by recording a well-known lemma
stating that every traceless matrix is unitarily similar to a hollow matrix;
see e.g.~\cite{HJ85}*{Chapter 2, Section 2.2, Problem 3} for a proof.

\begin{lemma}\label{lem:0diag}
Given $A\in M_d(\C)$ with $\Tr(A)=0$,
there exists a unitary $U\in \GL_d(\C)$ so that $U^{-1}AU$ has  only zeros on the diagonal.
\end{lemma}

\begin{proof}[Proof of Proposition \ref{prop:kerTr}]

Since 
\begin{equation}\label{eq:spanupq0diag}
\lspan\big\{u^p_q\mid p\in\{1,2,\ldots,d-1\},\,q\in\{0,1,\ldots,d-1\}\big\}
\end{equation}
equals the set of matrices in $M_d(\C)$ having all diagonal entries zero,
in light of Lemma~\ref{lem:0diag}, it will suffice to show that $\Image f$ contains the space~\eqref{eq:spanupq0diag}.
Using~\eqref{eq:sum0}, we have
\begin{align}
f(u^0_0,u^0_1,u^p_q)&=\big((a_{123}+a_{213}+a_{231})+(a_{132}+a_{312}+a_{321})\omega^p\big)u^p_{q+1} \notag \\
&=(a_{123}+a_{213}+a_{231})(1-\omega^p)u^p_{q+1} \label{eq:fuuusum1}
\end{align}
and, similarly,
\begin{align*}
f(u^0_1,u^0_0,u^p_q)&=(a_{123}+a_{132}+a_{213})(1-\omega^p)u^p_{q+1} \\
f(u^0_1,u^p_q,u^0_0)&=(a_{123}+a_{132}+a_{312})(1-\omega^p)u^p_{q+1}\,.
\end{align*}
Suppose that at least one of the three quantities
\begin{equation}\label{eq:sums3}
a_{123}+a_{213}+a_{231}\,,\qquad
a_{123}+a_{132}+a_{213}\,,\qquad
a_{123}+a_{132}+a_{312}
\end{equation}
is nonzero.
If the first is nonzero, then using~\eqref{eq:fuuusum1} and~\eqref{eq:spanupq0diag} we get
\begin{multline*}
f(u^0_0,u^0_1,\lspan\{u^p_q\mid p\in\{1,2,\ldots,d-1\},\,q\in\{0,1,\ldots,d-1\}\})= \\
\begin{aligned}
&=\lspan\{f(u^0_0,u^0_1,u^p_q)\mid p\in\{1,2,\ldots,d-1\},\,q\in\{0,1,\ldots,d-1\}\}= \\
&=\lspan\{u^p_{q+1}\mid p\in\{1,2,\ldots,d-1\},\,q\in\{0,1,\ldots,d-1\}\},
\end{aligned}
\end{multline*}
and the last of these is the set of all matrices having zero diagonal.
By Lemma \ref{lem:0diag},
the union of the similarity orbits (in fact, the unitary orbits) of these matrices is
$M_d(\C)\cap\ker\Tr$.
The cases of the  other expressions in~\eqref{eq:sums3} being nonzero are handled in a like manner.

Hence, we may suppose that all three expressions in~\eqref{eq:sums3} vanish.
Then, using also~\eqref{eq:sum0}, we get
\begin{equation}\label{eq:asolns1}
a_{213}=-a_{123}-a_{132}\,,\quad a_{231}=a_{132}\,,\quad
a_{312}=-a_{123}-a_{132}\,,\quad a_{321}=a_{123}\,.
\end{equation}
We find
\begin{align*}
f(u^0_1,u^p_q,u^0_1)
&=\big((a_{123}+a_{321})\omega^p+(a_{132}+a_{312})+(a_{213}+a_{231})\omega^{2p}\big)u^p_{q+2} \\
&=-a_{123}(1-\omega^p)^2u^p_{q+2},
\end{align*}
using~\eqref{eq:asolns1} for the second equality,
and similarly
\[
f(u^0_1,u^0_1,u^p_q)
=-a_{132}(1-\omega^p)^2u^p_{q+2}\,.
\]
Arguing as before, if $a_{123}\ne0$ or $a_{132}\ne0$, then
the set~\eqref{eq:spanupq0diag} is in $\Image f$.
But we cannot have $a_{123}=a_{132}=0$, for this together with~\eqref{eq:asolns1}
would imply that $f$ is identically zero, contrary to hypothesis.
\end{proof}


\end{document}